\documentclass[12pt]{amsart}
\usepackage{latexsym}
\usepackage{amssymb, amsmath,mathrsfs,amsthm}
\usepackage{geometry}
\usepackage[OT2,T1]{fontenc}
\usepackage{setspace}
\usepackage{mathtools}
\usepackage{enumerate}
\usepackage{float}

\newtheorem{theorem}{Theorem}[section]
\newtheorem*{theorem*}{Theorem}
\newtheorem{lemma}[theorem]{Lemma}

\newtheorem{corollary}[theorem]{Corollary}
\newtheorem{proposition}[theorem]{Proposition}
\newtheorem*{question}{Question}
\usepackage{color}

\DeclareSymbolFont{cyrletters}{OT2}{wncyr}{m}{n}
\DeclareMathSymbol{\Sha}{\mathalpha}{cyrletters}{"58}

\theoremstyle{remark}

\begin{document}
\raggedbottom

\title{Interpolation in Model Spaces}

\author{Pamela Gorkin}
\address{Pamela Gorkin, Department of Mathematics\\ Bucknell University\\  Lewisburg, PA  USA 17837}
\email{pgorkin@bucknell.edu}

\author{Brett D. Wick}
\address{Brett D. Wick, Department of Mathematics \& Statistics, Washington University in St. Louis, St. Louis, Missouri, USA}
\email{wick@math.wustl.edu}
%\thanks{Research supported in part by NSF grants DMS-1800057 and DMS-1560955, as well as ARC DP190100970.}

%\date{\today}
%\author{Pamela Gorkin$^\dagger$, and Brett Wick$^\ddagger$}
%\thanks{$\dagger$ Research supported in part by National Science Foundation (need to modify this).}
%\thanks{$\ddagger$ Research supported in part by National Science Foundation Grant }
%\address{Pamela Gorkin, Department of Mathematics, Bucknell University, 380 Olin Science Building, Lewisburg, PA 17837, USA.}
%\email{pgorkin@bucknell.edu}
%\address{Brett Wick, Department of Mathematics, Washington University in St. Louis, St. Louis, Missouri, USA.}
%\email{brett.wick@gmail.com}

\begin{abstract} 
In this paper we consider interpolation in model spaces, $H^2 \ominus B H^2$ with $B$ a Blaschke product. We study unions of interpolating sequences for two sequences that are far from each other in the pseudohyperbolic metric as well as two sequences that are close to each other in the pseudohyperbolic metric. The paper concludes with a discussion of the behavior of Frostman sequences under perturbations.
\end{abstract}
\renewcommand{\thefootnote}{\fnsymbol{footnote}} 
%\footnotetext{Primary ; Secondary }
\renewcommand{\thefootnote}{\arabic{footnote}}
%\keywords{model space}

\maketitle

\section{Introduction}

 Let $H^\infty$ denote the space of bounded analytic functions and let $H^2$ denote the Hardy space of functions on the unit circle $\mathbb{T}$ satisfying 
\[\sup_{0 < r < 1} \int_\mathbb{T} |f(r\zeta)|^2 dm(\zeta) < \infty.\]  A sequence $(a_j)$ of points in $\mathbb{D}$ is interpolating for $H^\infty$, if for every bounded sequence $(\alpha_j)$ of complex numbers, there is a function $f \in H^\infty$ with $f(a_j) = \alpha_j$ for all $j$. A Blaschke product $B$ with zero sequence $(a_j)$ is called an interpolating Blaschke product if its zero sequence is an interpolating sequence for $H^\infty$. Carleson's theorem tells us that the Blaschke product is interpolating if and only if there exists $\delta > 0$ with \[\inf_n (1 - |a_n|^2)|B^\prime(a_n)| \ge \delta.\] The main goal of this paper is to study unions of interpolating sequences that are near and far from each other in the setting of the model space $H^2 \ominus B H^2$ with $B$ a Blaschke product.

To set the context for our work requires some notation:
For an inner function $\theta$, let $K_\theta^2 := H^2 \ominus \theta H^2 = H^2 \cap \theta (\overline{zH^2})$, where $\overline{zH^2}$ denotes the set of functions with complex conjugate in $z H^2$.  We let $K_\theta^\infty = H^\infty \cap \theta \overline{z H^\infty} = H^\infty \cap \theta \overline{zH^2}$, and we let \[K_{*\theta} := K^2_\theta \bigcap BMO,\] 
where $BMO$ denotes the space of functions of bounded mean oscillation on the unit circle.

 For a sequence $(a_j)$ of points in the open unit disk $\mathbb{D}$ satisfying the Blaschke condition $\sum_j (1 - |a_j|) < \infty$, we consider Blaschke products, or functions of the form
\[B(z) = \lambda \prod_{j = 1}^\infty \frac{|a_j|}{-a_j}\left(\frac{z-a_j}{1 - \overline{a_j}z}\right)~\mbox{where}~\lambda \in \mathbb{T}.\]
(Here, as in the future, we interpret $|a_j|/a_j = 1$ if $a_j = 0$.) We are particularly interested in Blaschke products for which the zero sequence $(a_j)$ is an interpolating sequence for $H^\infty$.

In \cite{D}, Dyakonov proved the following:

\begin{theorem}[\cite{D}]\label{thm:D1} Suppose that $(\alpha_j)$ is an $\ell_\infty$ sequence and $B$ is an interpolating Blaschke product with zeros $(a_j)$. In order that there exist a function $f \in K_B^\infty$ for which $f(a_j) = \alpha_j$ for all $j$, it is necessary and sufficient that 
\begin{equation}\label{eqn:D}
\sup_k \left|\sum_j \frac{\alpha_j}{B^\prime(a_j)(1 - a_j \overline{a_k})}\right| < \infty.\end{equation}
\end{theorem}

\noindent
Note that Theorem~\ref{thm:D1} assumes only that $(a_j)$ can be interpolated to a particular sequence $(\alpha_j)$; in particular, one satisfying the conditions of equation \eqref{eqn:D}.

In this paper, we combine Dyakonov's techniques with those of Kenneth Hoffman to obtain further results about interpolation in $K_\theta^\infty$.

To discuss these results, we need a measure of separation of points in the open unit disk $\mathbb{D}$. The natural metrics are the hyperbolic or pseudohyperbolic distances. We begin with the latter.  Let \[\rho(a, z) = \left|\frac{z - a}{1 - \overline{a}z}\right|\] denote the pseudohyperbolic distance between two points $a$ and $z$ in $\mathbb{D}$.

If $(a_j)$ and $(z_j)$ are two sequences of points in $\mathbb{D}$ and we assume that we can interpolate $(a_j)$ to any $\ell_\infty$ sequence $(\alpha_j)$ and $(z_j)$ to any $\ell_\infty$ sequence $(\beta_j)$ then, using Hoffman's results, it is not difficult to show that $\rho$-separation of $(a_j)$ and $(z_j)$ implies that we can interpolate an (ordered) union of the sequences to any $\ell_\infty$ sequence. For ease of notation, we will primarily consider the union defined by alternating points of the sequences.

In this paper, we first consider the case when the sequences $(z_j)$ and $(a_j)$ are ``far from each other'': We show (Theorem~\ref{thm:rhoseparated}) that if $(a_j)$ can be interpolated to $(\alpha_j)$ in $K_B^\infty$ and $(z_j)$ can be interpolated to $(\beta_j)$ in $K_C^\infty$, then the union of the two sequences can be interpolated to the union of $(\alpha_j)$ and $(\beta_j)$ (in the appropriate order) in $K_{BC}^\infty$ if the sequences $(a_j)$ and $(z_j)$ are $\rho$-separated; that is, there exists a constant $\lambda > 0$ such that $\rho(a_j, z_k) \ge \lambda$ for all $j$ and $k$. Using Theorem~\ref{thm:D1} allows us to rephrase this as a statement about a series like the one appearing in equation \eqref{eqn:D}. 

We then consider two $\rho$-separated sequences $(a_j)$ and $(z_j)$ that are ``near each other''; that is, with the property that there exists $\lambda < 1$ with 
$\rho(a_j, z_j) < \lambda < 1$ for all $j$. In this case, we show that the modified statement of Proposition~\ref{prop:near} is true for sequences in model spaces (Theorem~\ref{thm:nearby}); that is, if $(a_n)$ is interpolating for $K_B^\infty$ and the two sequences are near each other, then $(z_n)$ is interpolating for $K_C^\infty$. 

From this result, we obtain some information about (uniform) {\it Frostman Blaschke products}. Recall that a sequence $(a_j)$ in $\mathbb{D}$ satisfies the Frostman condition if and only if
\begin{equation}
\label{eqn:fst}
\sup\left\{\sum_j \frac{1 - |a_j|}{|\zeta - a_j|} : \zeta \in \mathbb{T}\right\} < \infty. 
\end{equation} As a consequence of Vinogradov's work \cite{V}, it follows that an $H^\infty$-interpolating sequence $(a_j)$ is Frostman if and only if it is interpolating for $K_B^\infty$. This can also be seen by considering Theorem~\ref{thm:D1} and using the following: In \cite[Section 3]{C1}, Cohn shows that an interpolating sequence $(a_k)$ is a Frostman sequence if and only if 
\begin{equation}\label{eqn:Cohn}
\sup_n \sum_k \frac{1 - |a_k|}{|1 - \overline{a_k}a_n|} < \infty.
\end{equation} 

Our paper concludes with a fact about (uniform) Frostman Blaschke products that we have not seen in the literature. Recall that a Frostman Blaschke product is a Blaschke product with zeros $(a_n)$ that satisfy the Frostman condition \eqref{eqn:fst}. An example of such a Blaschke product appears in \cite{MR} (or \cite[p. 130]{CMR}) and is given by 
\[a_n = \left(1 - \frac{1}{2^n}\right)\exp\left(i \frac{2^n}{3^n}\right).\] In general, it is not easy to check that something is a Frostman Blaschke product. Vasyunin has shown that if $B$ is a uniform Frostman Blaschke product with zeros $(a_n)$, then $\sum_{n = 1}^\infty (1 - |a_n|)\log(1/(1-|a_n|)) < \infty$, but this is not a characterization. For generalizations of this as well as more discussion see \cite{AG}. Here, we show that if you start with a uniform Frostman Blaschke product and move the zeros, but not too far pseudo-hyperbolically speaking, then the resulting Blaschke product is also a uniform Blaschke product. In view of the difficulty of proving something is a Frostman Blaschke product, this result could be useful.  We accomplish this by using Dyakonov's methods and result to conclude that as long as we move the zeros of a Frostman Blaschke product within a fixed pseudohyperbolic radius $r < 1$ of the original zeros, the resulting Blaschke product will remain a Frostman Blaschke product.

 \section{Preliminaries}

In this section we collect all the necessary background and estimates that play a role in the proofs in later sections. We first recall the fact that if points are close to an interpolating sequence, then they are interpolating as well.

\begin{proposition}\label{prop:near}\cite[p. 305]{G}  Let $(a_j)$ be an interpolating sequence for $H^\infty$ and $(z_j)$ a $\rho$-separated sequence with 
\[\rho(a_j, z_j) < \lambda < 1,\] for all $j$, then $(z_j)$ is an interpolating sequence  for $H^\infty$.
\end{proposition}

This proposition is an exercise in \cite{G}. For a proof, see \cite[Theorem 27.33]{M}. Using the same notation as above, we will need the following estimate that appears in the proof:

\begin{equation}\label{eqn:pseudo}
1 - \rho(a_j, a_k) \le \left(\frac{1 + \lambda}{1 - \lambda}\right)^2 (1 - \rho(z_j, z_k)).
\end{equation}

Recall that for two points $z$ and $w$ in $\mathbb{D}$,  the pseudohyperbolic distance is $\rho(z, w) = \left|\frac{z - w}{1 - \overline{w}z}\right|$ and the hyperbolic metric is given by  \[\beta(z, w) = \frac{1}{2} \log \frac{1 + \rho(z, w)}{1 - \rho(z, w)}.\] 

In what follows, we will consider two interpolating sequences $(a_j)$ and $(z_j)$ that are $\rho$-separated or {\it far from each other}; that is, with the property that there exists $\varepsilon > 0$ with
\begin{equation}\label{eqn:bounded}
\inf_{j, k} \rho(a_j, z_k) \ge \varepsilon.
\end{equation} 
We then consider sequences that are {\it near each other} in the sense that there exists $\varepsilon < 1$ with 
$\rho(a_j, z_j) < \varepsilon < 1$ for all $j$. In this case, we have the following estimates that we will refer to later.  Let $\varepsilon$ be chosen with $0 < \varepsilon < 1$.  Suppose that $\rho(a_j, z_j) \le \varepsilon$ for all $j$. Then $r:=\sup_{j, k} \beta(a_j, z_k) \le \frac{1}{2} \log \frac{2}{1 - \varepsilon} < \infty$. Let $s = \tanh r \in (0, 1)$ and apply (\cite[Proposition 4.5]{Zhu}) to obtain for each $j$ and $k$,

\begin{equation}\label{eqn:inequality1} 1 - s \le \frac{1 - s|z_k|}{1 - |z_k|^2} \le \frac{1}{|1 - \overline{a_j} z_k|}~\mbox{and}~\frac{1 - s|a_j|}{1 - |a_j|^2} \le \frac{1}{|1 - \overline{a_j} z_k|}.\end{equation}

Thus, 
\[\frac{1 - |z_k|^2}{|1 - \overline{a_j}z_k|} \ge 1 - s |z_k| \ge 1 - \tanh r,\] and a similar inequality holds with $z_j$ replaced by $a_j$.

Our work relies on Dyakonov's proof techniques, which rely on the following two results of W. Cohn. The convergence below is taken in the weak-$\ast$ topology of $BMOA := BMO \cap H^2$, and it also converges in $H^2$. Thus, the convergence also holds on compact subsets of $\mathbb{D}$.  

\begin{lemma}\label{lem:1} 
\cite[Lemma 3.1]{C}
Given an interpolating Blaschke product $B$ with zeros $(a_j)$, the general form of a function $g \in K_{*B}$ is
\[g(z) = \sum_j c_j \frac{1 - |a_j|^2}{1 - \overline{a_j}z},\] where $(c_j) \in \ell_\infty$.
\end{lemma}

\begin{lemma}\label{lem:Cohn1}
\cite[Corollary 3.2]{C1}
Let $B$ be an interpolating Blaschke product with zeros $(a_j)$ and let $g \in K_{*B}$. Then $(g(a_j))\in \ell_\infty$ if and only if $g \in H^\infty$.
 \end{lemma}

Another key ingredient in our proofs are the following three theorems from Kenneth Hoffman's seminal paper, which we recall here.  

\begin{lemma}[Hoffman's Lemma] \cite{H}, \cite[p. 395]{G} Suppose that $B$ is an interpolating Blaschke product with zeros $(z_n)$ and 
\[\inf_n(1 - |z_n|^2)|B^\prime(z_n)| \ge \delta > 0.\]
Then there exist $\lambda := \lambda(\delta)$ with $0 < \lambda < 1$ and $r:=r(\delta)$ with $0 < r < 1$ satisfying 
\[\lim_{\delta \to 1} \lambda(\delta) = 1~\mbox{and}~\lim_{\delta \to 1} r(\delta) = 1\] such that
\[\{z: |B(z)| < r\}\] is the union of pairwise disjoint domains $V_n$ with $z_n \in V_n$ and
\[V_n \subset \{z: \rho(z, z_n) < \lambda\}.\]
\end{lemma}

Let $M(H^\infty)$ denote the maximal ideal space of $H^\infty$ or the set of non-zero multiplicative linear functionals on $H^\infty$. Identifying points of $\mathbb{D}$ with point evaluation, we may think of $\mathbb{D}$ as contained in $M(H^\infty)$. Carleson's Corona Theorem tells us that $\mathbb{D}$ is dense in the space in the weak-$\ast$ topology.  The maximal ideal space breaks down into analytic disks called Gleason parts. These may be a single point, in which case we call them trivial, or they may be true analytic disks, in which case we call them nontrivial. It is a consequence of Hoffman's work that points in the closure of an interpolating sequence are nontrivial. (See \cite[Theorem 4.3]{H}.)

\begin{theorem}\label{thm:H1} 
\cite[Theorem 5.3]{H}
Let $B$ be a Blaschke product and let $m$ be a point of $M(H^\infty) \setminus \mathbb{D}$ for which $B(m) = 0$. Then either $B$ has a zero of infinite order at $m$ or else $m$ lies in the closure of an interpolating subsequence of the zero sequence of $B$. \end{theorem}

In the same paper of Hoffman, \cite[Theorem 5.4]{H}, shows that an interpolating Blaschke product cannot have a zero of infinite order. Therefore, if $B$ is an interpolating Blaschke product and $B(m) = 0$, then $m$ must lie in the closure of the zero sequence of $B$.

\begin{theorem}[Hoffman's Theorem] \label{thm:H}
A necessary and sufficient condition that a point $m$ of the maximal ideal space lie in a nontrivial part is the following: If $S$ and $T$ are subsets of the disk $\mathbb{D}$ and if $m$ belongs to the closure of each set, then the hyperbolic distance from $S$ to $T$ is zero.
\end{theorem}

As a result of Hoffman's theorem we show that, if $(a_j)$ is interpolating for $K_B^\infty$ and $(z_j)$ is interpolating for $K_C^\infty$ and the $\rho$ distance between the two sequences is positive, then (see Corollary \ref{cor:H}) $B$ is bounded below on $\{z_j\}$ and $C$ is bounded below on $\{a_j\}$. 
This is known, but for future use we isolate this as a corollary to Theorem~\ref{thm:H}.

\begin{corollary}\label{cor:H} Let $(a_j)$ and $(z_j)$ be two interpolating sequences for $H^\infty$ with corresponding Blaschke products $B$ and $C$, respectively. Suppose further that the $\rho$-distance between the two sequences satisfies \[\inf_{j, k}\rho(a_j, z_k) \ge \varepsilon > 0.\] Then there exists $\eta > 0$ such that 
\[\inf_j |C(a_j)| \ge \eta~\mbox{and}~\inf_j|B(z_j)| \ge \eta.\]
\end{corollary}

\begin{proof} If not, we may suppose that $\inf_j |B(z_j)| = 0$. Therefore, there exists a subsequence $(z_{j_k})$ of $(z_j)$ with $B(z_{j_k}) \to 0$. Let $m \in M(H^\infty) \setminus \mathbb{D}$ be a point in the closure of the set $\{z_{j_k}\}$. Then $B(m) = 0$. By the aforementioned work of Hoffman, $m$ lies in the closure of the zeros of $B$, namely the closure of $\{a_j\}$. On the other hand, $m$ lies in the closure of $\{z_j\}$,  by the choice of $m$. By Theorem~\ref{thm:H} the hyperbolic distance between the two sets must be zero. But since the pseudohyperbolic distance between the two is bounded away from zero, this is impossible.
\end{proof}

\section{Sequences that are far from each other}

In this section, we will consider unions of finitely many interpolating sequences defined in the following manner:
Let $(\alpha_j)$ and $(\beta_j)$ be sequences. Define $(\alpha_j) \cup (\beta_j)$ to be the sequence $(\gamma_j)$ where 
\begin{equation}\label{union}
\gamma_j = \left\{
	\begin{array}{ll}
		\alpha_j  & \mbox{if } j ~\mbox{is odd} \\
		\beta_j & \mbox{if } ~\mbox{is even}.
	\end{array}
\right.\end{equation}
For simplicity of presentation, we have defined the sequence $(\gamma_j)$ via this simple ``every-other'' interlacing.  It is clear that from the proof techniques that one could interlace the sequences $(\alpha_j)$ and $(\beta_j)$ in other ways.  Interlacing in other more exotic ways would necessitate the introduction of additional more complicated notation and to present the ideas most clearly we have chosen to use only these simple process described here.

In what follows, for a Blaschke product $B$ with zeros $(a_j)$, let \[b_j(z) = \frac{|a_j|}{-a_j}\frac{(z-a_j)}{(1 - \overline{a_j}z)}\] and let $B_j(z) = B(z)/b_j(z)$. (We interpret $\frac{|a_j|}{-a_j}=1$ if $a_j = 0$.)

If we wish to interpolate $(a_j) \cup (z_j)$ (as defined in equation \eqref{union}) to the sequence $(\alpha_j) \cup (\beta_j)$ and we know that  $(a_j)$ is interpolating for $K_B^\infty$ and $(z_j)$ is interpolating for $K_C^\infty$, and both $B(z_j)$ and $C(a_j)$ are bounded below over all $j$, then we can interpolate to $(\alpha_j^\prime):=(\alpha_j/C(a_j))$ and $(\beta_j^\prime):=(\beta_j/B(z_j))$ with $g_1 \in K_B^\infty$ and $g_2 \in K_C^\infty$, respectively. So $G:=C g_1 + B g_2 \in K_{BC}^\infty$ will do the interpolation. However, if we don't know that we can do the interpolation to every bounded sequence, then we need to combine Dyakonov and Hoffman's work to obtain a result. 
 
\begin{theorem}\label{thm:rhoseparated} Let $B$ and $C$ be interpolating Blaschke products with zeros $(a_j)$ and $(z_j)$ respectively,  satisfying $\displaystyle\inf_{j, k} \rho(z_j, a_k) \ge \varepsilon > 0$. If $(a_j)$ can be interpolated to $(\alpha_j)$ in $K_B^\infty$ and $(z_j)$ can be interpolated to $(\beta_j)$ in $K_C^\infty$, then $(x_j) := (a_j) \cup (z_j)$ can be interpolated to $(\gamma_j):= (\alpha_j) \cup (\beta_j)$ in $K_{BC}^\infty$. \end{theorem}

\begin{proof}  By Hoffman's theorem the sequence $(x_j)$ is interpolating for $H^\infty$ and $BC$ is interpolating. Note that 
\[K_B^\infty = H^\infty \cap B \overline{z H^\infty} = H^\infty \cap BC \overline{z C H^\infty} \subseteq K_{BC}^\infty\] and, similarly, $K_C^\infty \subseteq K_{BC}^\infty$. Corollary~\ref{cor:H} also implies that there exists $\delta > 0$ such that $\displaystyle\inf_j (\min\{|B(z_j)|, |C(a_j)|\}) \ge \delta > 0$. We define $\tilde{\gamma_j}$ for $j = 1, 2, 3, \ldots$, by  

\[
\tilde{\gamma}_{2j-1} := \frac{-\overline{a_j}}{|a_j|}\frac{\overline{\gamma_{2j-1}}}{\overline{B_j(a_j)}} ~\mbox{and}~ 
		\tilde{\gamma}_{2j} \\
		= \frac{-\overline{z_j}}{|z_j|} \frac{\overline{\gamma_{2j}}}{\overline{C_j(z_j)}}.
\]

Then $(\tilde{\gamma}_j) \in \ell_\infty$. Let $g$ be defined by 

\[g(z) = \sum_{j = 1}^\infty \tilde{\gamma}_{2j -1} \frac{(1 - |a_{j}|^2)}{1 - \overline{a_{j}}z} + \sum_{j = 1}^\infty \tilde{\gamma}_{2j} \frac{(1 - |z_j|^2)}{1 - \overline{z_j}z},\] 
and use Lemma~\ref{lem:1}, the fact that $BC$ is interpolating, and $K_B^2 \cup K_C^2 \subseteq K_{BC}^2$ to conclude that $g \in K_{*BC}$. In particular, $g \in H^2$.

Now for almost every $z \in \mathbb{T}$, we have
\begin{equation}\label{eqn:BC} B(z) C(z) \overline{zg(z)}=  \sum_j \overline{\tilde{\gamma}_{2j-1}}B(z)C(z) \overline{z} \frac{(1 - |a_{j}|^2)}{1 - {a_{j}}\overline{z}} +
\sum_j \overline{\tilde{\gamma}_{2j}}B(z) C(z) \overline{z} \frac{(1 - |z_{j}|^2)}{1 - {z_{j}}\overline{z}}.\end{equation}
For the first summand and almost every $z \in \mathbb{T}$, 

\begin{multline*}
B(z)C(z) \overline{z} \frac{(1 - |a_{j}|^2)}{1 - {a_{j}}\overline{z}} = \\ \frac{-|a_j|}{a_j} C(z) B_j(z) \left(\frac{z - a_j}{1 - \overline{a_{j}} z} \right)\overline{z}  \frac{(1 - |a_{j}|^2)}{1 - \overline{z}a_{j}} = \frac{-|a_j|}{a_j}C(z) B_j(z) \frac{(1 - |a_{j}|^2)}{1 - \overline{a_{j}}z}.
\end{multline*}

The summation converges in $H^2$ and each summand is in $H^2$, so the function also lies in $H^2$. Therefore 

\[BC\overline{zg(z)} \in \left(BC\overline{zH^2}\right) \bigcap H^2 = K_{BC}^2.\]
 
 The same computations, with appropriate adjustments, hold for the second summand. Therefore,
 
\begin{multline}\label{eqn:G}
G(z):=B(z)C(z)\overline{zg(z)}= \underbrace{\sum_{j=1}^\infty \overline{\tilde{\gamma}_{2j-1}}\left(\frac{-|a_j|}{a_j} B_j(z)C(z) 
\frac{(1 - |a_{j}|^2)}{1 - \overline{a_{j}}z}\right)}_{G_1} +\\
\underbrace{\sum_{j=1}^\infty \overline{\tilde{\gamma}_{2j}}\left(\frac{-|z_j|}{z_j}B(z) C_j(z) \frac{(1 - |z_{j}|^2)}{1 - \overline{z_{j}}z}\right)}_{G_2} \in K_{BC}^2.\end{multline}

Note that the equality 
\[G(z):= \sum_{j=1}^\infty \frac{-|a_j|}{a_j}\overline{\tilde{\gamma}_{2j-1}}B_j(z)C(z) \frac{(1 - |a_{j}|^2)}{1 - \overline{a_{j}}z} +\\
\sum_{j=1}^\infty \frac{-|z_j|}{z_j}\overline{\tilde{\gamma}_{2j}}B(z) C_j(z) \frac{(1 - |z_{j}|^2)}{1 - \overline{z_{j}}z}\] also holds in $\mathbb{D}$.

Now from \eqref{eqn:BC} $G_1 \in H^2 \cap BC(\overline{zH^2}) =  K_{BC}^2$, $G_1 \in CH^2$, and
\[G_1(a_j) = \alpha_j C(a_j).\] But we assume there exists $f \in K_B^\infty$ with $f(a_j) =  \alpha_j$  for all $j$, and therefore $(Cf)(a_j) = \alpha_j(C(a_j))$.  It follows that
\[G_1 - Cf = Bh~\mbox{for some}~h \in H^2.\] But since $G_1 \in CH^2$ and $B$ and $C$ have no common zeros, we see that $C$ must divide $h$. Thus, we have $G_1 - Cf = BCh_1$ for some $h_1 \in H^2$. Thus $G_1 - Cf \in BCH^2$. Note also that $f \in K_B^\infty$ implies that $Cf \in K_{BC}^\infty$. So, 
\[G_1 - Cf \in (BC)H^2 \cap K_{BC}^2 = \{0\}.\] Therefore, $G_1 = 
Cf \in H^\infty$.   
 The same computations show that $G_2 \in H^\infty$. Therefore $G = G_1 + G_2 \in H^\infty \cap K_{BC}^2$, which implies the result.
\end{proof}

From Theorem~\ref{thm:D1} we have the following:

\begin{corollary} Let $B$, $C$, $(a_j)$, and $(z_j)$ be as in Theorem~\ref{thm:rhoseparated} and let $(x_j) = (a_j) \cup (z_j)$, where $(a_j)$ and $(z_j)$ are the zeros of $B$ and $C$, respectively. If
\[\sup_k \left|\sum_j \frac{\alpha_j}{B^\prime(a_j)(1 - a_j \overline{a_k})} \right| < \infty ~\mbox{and}~ 
 \sup_k \left|\sum_j \frac{\beta_j}{C^\prime(z_j)(1 - z_j \overline{z_k})} \right| < \infty,\]

\[ ~\mbox{then}~ \sup_k \left|\sum_j \frac{\gamma_j}{(BC)^\prime(\alpha_j)(1 - \alpha_j \overline{\alpha_k})} \right| < \infty\]
where $(\gamma_j)=(\alpha_j)\cup(\beta_j)$.
\end{corollary}

\section{Sequences that are near each other}

In the introduction to the paper, we mentioned (see Proposition~\ref{prop:near}) that if $(a_n)$ is an interpolating sequence for $H^\infty$ and $(z_n)$ is a $\rho$-separated sequence with $\rho(a_n, z_n) < 1 - \varepsilon < 1$ for all $n$, then $(z_n)$ is interpolating for $H^\infty$. Here we consider the same result for $K_B^\infty$.

 \begin{proposition}\label{prop:C} If $(a_n)$ is interpolating for $K_B^\infty$, then there is a constant $M$ such that  $\|f\|_\infty \le M\|(f(a_n))\|_\infty$ for every $f \in K_B^\infty$.  \end{proposition}

\begin{proof} Define $T: K_B^\infty \to \ell_\infty$ by $T(f) = (f(a_n))$. Then $T$ is a bounded linear operator that maps surjectively onto $\ell_\infty$. Note that $T$ is also injective, because $T(f) = T(g)$ implies $f - g \in BH^\infty$. But $f-g \in K_B^\infty \cap BH^\infty$ implies that $f = g$. The desired result now follows from the open mapping theorem (or, more specifically, the bounded inverse theorem).
\end{proof}

We now prove that when points in an interpolating sequence for $K_B^\infty$ can be moved pseudohyperbolically, as long as they are not moved too far, the new sequence will be interpolating for $K_B^\infty$ if the original was. 

\begin{theorem}\label{thm:nearby} Let $B$ be a Blaschke product and suppose that its zero sequence, $(a_n)$, is an interpolating sequence for $K_B^\infty$. Let $M$ be the constant in Proposition~\ref{prop:C}, and suppose that $(z_n)$ is a sequence of distinct points with $\rho(a_n, z_n) < 1 - \varepsilon < 1/(2M)$. Then $(z_n)$ is interpolating for $K_B^\infty$.
\end{theorem}

\begin{proof}  Without loss of generality we may assume $M > 1$. Let $(\alpha_n) \in \ell_\infty$. Choose $f_0 \in K_B^\infty$ with $f_0(a_n) = \alpha_n$ for all $n$. If necessary, divide $(\alpha_n)$ by a constant to assume that we can choose $f_0$ with norm at most one. Then, by Schwarz's lemma, for all $n$ we have \[\rho(f_0(z_n), f_0(a_n)) \le \rho(z_n, a_n).\] Thus,
\[|f_0(z_n) - f_0(a_n)| \le 2\rho(a_n, z_n),\] for all $n$.

 So, using our assumptions, for all $n$ we have
\[|f_0(z_n) - f_0(a_n)| \le 2(1 - \varepsilon).\]
Now $(a_n)$ is interpolating for $K_B^\infty$, so we may choose $f_1\in K_B^\infty$ so that $f_1(a_n) = f_0(a_n) - f_0(z_n)$ for all $n$.  By Proposition~\ref{prop:C}, we know that \[\|f_1\|_\infty \le M \|(f_1(a_n))\|_\infty \le 2M(1-\varepsilon).\] Therefore, by Schwarz's lemma, 
\[\left|\left(\frac{1}{2M(1-\varepsilon)}\right)f_1(z_n) - \left(\frac{1}{2M(1-\varepsilon)}\right)f_1(a_n)\right| \le 2\rho(a_n, z_n).\] Consequently, for all $n$ we have
\[|f_1(z_n) + f_0(z_n) - f_0(a_n)| = |f_1(z_n) - f_1(a_n)| \le 2^2 M (1 - \varepsilon)^2.\]
Now we choose $f_2 \in K_B^\infty$ with 
\[\|f_2\| \le M\|(f_1(z_n) - f_1(a_n))\|_\infty\]
and
\[f_2(a_n) = -(f_1(z_n) - f_1(a_n)) = f_0(a_n) - (f_1(z_n) + f_0(z_n)).\]
Therefore
\[\|f_2\| \le 2^2 M^2 (1 - \varepsilon)^2~\mbox{and}~f_2(a_n) = -(f_1(z_n) + f_0(z_n) - f_0(a_n)).\]
Now by Schwarz's lemma we have 
\[\left|f_2(z_n)/\|f_2\| - f_2(a_n)/\|f_2\|\right| \le 2  \rho(a_n, z_n) \le 2 (1 - \varepsilon),\] and consequently
\[|f_2(z_n) + f_1(z_n) + f_0(z_n) - f_0(a_n)| = |f_2(z_n) - f_2(a_n)|  \le 2^3 M^2 (1 - \varepsilon)^3.\] 
Continuing in this way, we assume we have chosen $f_0, \ldots, f_m \in K_B^\infty$ with
\[|f_m(z_n) + \cdots + f_0(z_n) - \alpha_n| \le 2^{m+1}M^m (1 - \varepsilon)^{m+1}~\mbox{for all}~ n\] and 
\[\|f_m\| \le 2^m M^m (1 - \varepsilon)^{m}.\]  We choose $f_{m+1} \in K_B^\infty$ with 
\[f_{m+1}(a_n) = -(f_m(z_n) + \cdots + f_0(z_n) - \alpha_n) ~\mbox{and}~ \|f_{m+1}\| \le 2^{m+1}M^{m+1}(1 - \varepsilon)^{m+1}.\] Now we have chosen $\varepsilon$ so that $(1 - \varepsilon) < 1/(2M)$ and $\|f_{m+1}\| \le \left(2M(1 - \varepsilon)\right)^{m+1}$. Letting $f = \sum_{j = 0}^\infty f_j$ we obtain $f \in K_B^\infty$ with the property that for each $n$

\begin{multline*}
|f(z_n) - \alpha_n| = \lim_m |f_m(z_n) + \cdots + f_1(z_n) + f_0(z_n) - f_0(a_n)| \le \\
\lim_m 2^{m+1} M^m(1 - \varepsilon)^{m+1} = 0. \end{multline*}
Thus $f \in K_B^\infty$ and $f$ does the interpolation.

\end{proof}

\section{Frostman Blaschke products and sequences that are near each other}

Tolokonnikov \cite{T1} showed that Frostman Blaschke products are always a finite product of interpolating Blaschke products, \cite{T}. In view of this, if we start with two sequences $(a_n)$ and $(z_n)$ with $\rho(a_n, z_n) \le \lambda < 1$ for all $n$ and $(a_n)$ a Frostman sequence, then we can write $(a_n)$ as a finite union of interpolating sequences and, as long as $(z_n)$ is $\rho$-separated, the corresponding subsequences of $(z_n)$ will also be interpolating, by Proposition~\ref{prop:near}. For this reason, we can reduce our discussion to Frostman sequences that are interpolating for $H^\infty$.

\begin{proposition}\label{prop} Let $(a_n)_{n \in \mathbb{N}}$ be a sequence of points in $\mathbb{D}$. If $N$ is an integer for which $(a_n)_{n > N}$ is a Frostman sequence, then $(a_n)$ is a Frostman sequence.\end{proposition}

\begin{proof} Consider the function $F(\zeta): = \sum_{j = 1}^{N} \frac{1 - |a_j|^2}{|a_j - \zeta|}$ on the unit circle. Then $F$ is continuous and therefore bounded. Thus, $\displaystyle\sup_{\zeta \in \mathbb{T}}  \sum_{j = 1}^{N} \frac{1 - |a_j|^2}{|a_j - \zeta|}$ is finite and the result follows.
\end{proof}

We turn to the main theorem of this section, which says that if we begin moving points of a Frostman sequence, as long as we don't move the sequence too far pseudohyperbolically, the new sequence will be interpolating for $K_C^\infty$, where $C$ is the Blaschke product corresponding to the new sequence.

\begin{theorem}\label{movingzeros} Let $\varepsilon > 0$. Let $(a_n)$ be an interpolating Frostman sequence and let $(z_n)$ be  a $\rho$-separated sequence with  $\rho(a_n, z_n) \le 1 - \varepsilon$ for all $n$. Then $(z_n)$ is a Frostman sequence.\end{theorem}

\begin{proof} Using Proposition~\ref{prop:near} and \eqref{eqn:pseudo} we know that for all $j$ and $k$,
\[1 - \rho(a_j, a_k) \le \left(\frac{1 + (1-\varepsilon)}{1 - (1 - \varepsilon)}\right)^2 \left(1 - \rho(z_j, z_k)\right).\]  %Because both sequences are separated, there exist $\delta_B$ such that $\delta_B \le \rho(a_j, a_k) \le 1$ for all $j \ne k$ and $\delta_C$ such that $\delta_C \le \rho(z_j, z_k) \le 1$ for all $j \ne k$. 
Since $(1 + \rho(a_j, a_k)) \le 2$ and $1 \le 1 + \rho(z_j, z_k)$, it follows that 

\[(1+ \rho(a_j, a_k))(1 - \rho(a_j, a_k)) \le 2 \left(\frac{1 + (1-\varepsilon)}{1 - (1 - \varepsilon)}\right)^2 (1 + \rho(z_j, z_k))\left(1 - \rho(z_j, z_k)\right).\] A computation shows that
\[1 - \rho^2(a_j, a_k) = \frac{(1 - |a_j|^2)(1 - |a_k|^2)}{|1 - \overline{a_j}a_k|^2}.\] Since all of this also holds with the roles of $(a_j)$ and $(z_j)$ interchanged, there are positive constants $C_1=C_1(\varepsilon)$ and $C_2=C_2(\varepsilon)$ such that

\begin{equation}\label{eqn:4}
C_1 \frac{(1 - |a_j|^2)(1 - |a_k|^2)}{|1 - \overline{a_j}a_k|^2} \le \frac{(1 - |z_j|^2)(1 - |z_k|^2)}{|1 - \overline{z_j}z_k|^2} \le C_2\frac{(1 - |a_j|^2)(1 - |a_k|^2)}{|1 - \overline{a_j}a_k|^2}. 
\end{equation} 

Now $\rho(a_n, z_n) < 1 - \varepsilon := r$ and we know that every pseudohyperbolic disk is a Euclidean disk (see \cite{G}, Chapter 1). If we rotate the disk by $\alpha_n$, where $\alpha_n := |a_n|/a_n$ (interpreting $\alpha_n = 1$ if $a_n = 0$), we do not change pseudohyperbolic distances; that is, for $a, z \in \mathbb{D}$ and $\alpha \in \mathbb{T}$,
\[\rho(\alpha a, \alpha z) = \rho(a, z).\] So, $\alpha_n z_n \in D_\rho(|a_n|, 1 - \varepsilon)$. Now we use the fact that the pseudohyperbolic disk $D_\rho(|a_n|, 1-\varepsilon)$ is a Euclidean disk centered at the real number
\[p_n = \frac{1-r^2}{1-r^2|a_n|^2}|a_n| \in \mathbb{R}\] with radius \[R_n = \frac{1-|a_n|^2}{1 - r^2|a_n|^2}r.\]

Since $|a_n| \to 1$, there are finitely many $a_n$ for which $|a_n| \le 1 - \varepsilon$ and finitely many corresponding $z_n$. If we show that the Blaschke product with zeros $(z_n)_{n \ge N}$ is a Frostman Blaschke product, then Proposition~\ref{prop} implies that the Blaschke product with zeros $(z_n)_{n}$ is also a Frostman Blaschke product. Thus, we may assume that, for all $n$. we have $|a_n| \ge 1 - \varepsilon = r$ and $|z_n| \ge 1 - \varepsilon$.

The assumption that $r = 1 - \varepsilon < |a_n| = \rho(0, |a_n|)$, implies that $0$ is not in $D_\rho(|a_n|, r)$ for all such $a_n$, and therefore $0$ is not in the Euclidean disk $D(p_n, R_n)$. Since  $\alpha_n z_n \in D_\rho(|a_n|, r) = D(p_n, R_n)$, and $D(p_n, R_n)$ is a Euclidean disk with center on the positive real line, all points in $D(p_n, R_n)$ have modulus greater than $p_n - R_n$. A computation shows that
\[p_n - R_n = \frac{|a_n| - r^2|a_n|-r+r|a_n|^2}{1-r^2|a_n|^2} = \frac{(|a_n|-r)(1 + r|a_n|)}{1-r^2|a_n|^2} = \frac{|a_n|-r}{1-r|a_n|}.\] Since we assume that $|a_n| > r$ we have $p_n - R_n = \rho(|a_n|, r)$.

Thus,
$|z_n| = |\alpha_n z_n| \ge p_n - R_n = \rho(|a_n|, r)$. So
\[1 - |z_n|^2  \le  1 - \rho^2(|a_n|, r).\] 
Consequently, 
\begin{equation}\label{eqn:num}
1 - |z_n|^2 \le \frac{(1 - r^2)(1 - |a_n|^2)}{(1 - r |a_n|)^2}  \le \frac{1+r}{1-r} (1 - |a_n|^2).
\end{equation} Thus, for $C_r := \frac{1+r}{1-r}$ we have 
\[1 - |z_n|^2 \le C_r(1 - |a_n|^2),\] for all $n$ and we note that $C_r$ is a constant depending on $r$ but independent of $n$.
Similarly, since $\rho(a_n, z_n) < r$, we may interchange the roles of $a_n$ and $z_n$ above to see that $1 - |a_n|^2 \le C_r(1 - |z_n|^2)$, where $C_r$ is a  constant depending only on $r$ (and, hence, only on $\varepsilon$).

From the work above, we see that  $(1 - |a_m|^2) \asymp (1 - |z_m|^2)$; that is, there are positive constants $D_1$ and $D_2$ independent of $m$ with
\begin{equation}\label{eqn:equiv} D_1(1-|a_m|^2) \le 1 - |z_m|^2 \le D_2(1 - |a_m|^2)~\mbox{for all}~m.\end{equation}

Now, for all $z \in \mathbb{D}$ and all $j$ (see \cite[p. 4]{G})
\[\rho(a_j, z) \le \frac{\rho(a_j, z_j) + \rho(z_j, z)}{1 + \rho(a_j, z_j) \rho(z_j, z)}.\] Thus,
\[1 - \rho^2(a_j,z) \ge 1-\left(\frac{\rho(a_j, z_j) + \rho(z_j, z)}{1 + \rho(a_j, z_j) \rho(z_j, z)}\right)^2.\] 
Simplifying, we have
\[\frac{(1-|a_j|^2)(1-|z|^2)}{|1 - \overline{a_j}z|^2} \ge \frac{(1 - \rho^2(a_j, z_j))(1 - \rho^2(z_j, z))}{(1 + \rho(a_j, z_j)\rho(z_j, z))^2}.\]
Thus,
\[\frac{1-|a_j|^2}{|1 - \overline{a_j}z|^2} \ge \left(\frac{1 - \rho^2(a_j, z_j)}{\left(1 + \rho(a_j, z_j)\rho(z_j,z)\right)^2}\right)\left(\frac{1-|z_j|^2}{|1-\overline{z_j}z|^2}\right).\]
But by assumption $\rho(a_j, z_j) \le r < 1$ for all $j$, so 
\[\frac{1-|a_j|^2}{|1 - \overline{a_j}z|^2}  \ge \frac{(1- r^2)}{4}\left(\frac{1-|z_j|^2}{|1-\overline{z_j}z|^2}\right).\]
By equation~\eqref{eqn:equiv}, we have
\[\frac{1 - |z_j|^2}{|1 - \overline{a_j}z|^2} \ge D_1 \frac{(1-r^2)}{4}\left(\frac{1-|z_j|^2}{|1-\overline{z_j}z|^2}\right).\]
Therefore, for all $j$
\[\frac{1}{|1 - \overline{a_j}z|^2} \ge D_1\frac{(1-r^2)}{4}\left(\frac{1}{|1-\overline{z_j}z|^2}\right).\]
So there is a positive constant $C_3$, independent of $j$, such that for all $z \in \mathbb{D}$
 \[\frac{1}{|1 - \overline{a_j}z|} \ge C_3 \left(\frac{1}{|1-\overline{z_j}z|}\right).\]
Choose $\zeta \in \mathbb{T}$ and let $z \to \zeta$. Then 
\begin{equation}\label{eqn:den}
\frac{1}{|1 - \overline{a_j}\zeta|} \ge \frac{C_3}{|1-\overline{z_j}\zeta|}.\end{equation}
Since this holds for all $\zeta \in \mathbb{T}$, combining \eqref{eqn:num} and \eqref{eqn:den}, we see that there is a constant $C_4$ such that for all $j$,
\[\frac{1 - |a_j|^2}{|1 - \overline{a_j}\zeta|} \ge C_4\frac{1 - |z_j|^2}{|1 - \overline{z_j}\zeta|}.\] Thus, if $(a_j)$ is Frostman, so is $(z_j)$ and since this holds with the roles of $a_j$ and $z_j$ reversed, we have $(z_j)$ Frostman if and only if $(a_j)$ is Frostman.
\end{proof}

We note that the proof can be slightly shortened by using the characterization of Frostman sequences due to Cohn that appears in \eqref{eqn:Cohn}.
Since we can also obtain it directly, we prefer to do so.

%\begin{equation}\label{eqn:5}
%\frac{D_1}{|1-\overline{a_j}{a_k}|} \le \frac{1}{|1-\overline{z_j}z_k|} \le \frac{D_2}{|1 - \overline{a_j}a_k|}.
%\end{equation}

%But we know that $(1 - |a_n|^2)$ is equivalent to $(1-|z_n|^2)$ for every $n$. So, $\frac{1-|a_j|^2}{|1 - \overline{a_j} a_k|}$ is equivalent to $\frac{1-|z_j|^2}{|1 - \overline{z_j} z_k|}$. 

%Thus, we have the existence of positive constants $C$ and $D$ with 
%\[C|1-\overline{a_j}a_k| \le |1 - \overline{z_j}z_k| \le D|1 - \overline{a_j} a_k|.\] The constants are independent of $j$ and $k$. So let $\xi$ be a singularity of $B$. We should first show that the singularities of $B$ (corresponding to $(a_k)$) and $C$ (corresponding to $(z_k)$) are the same.  Then there is a subsequence of $(a_k)$ tending to $\xi$. Along this sequence, $z_{k_j}$ also tends to $\xi$ (insert argument). Therefore
%\[C |1 - \overline{a_j} \xi| \le |1 - \overline{z_j} \xi| \le D|1 - \overline{a_j} \xi|.\]
%Thus, at the singularities of $B$, which are the singularities of $C$, we have
%\[\sum_{j = 1}^\infty \frac{1 - |z_j|^2}{|\xi - z_j|} \le (C_r/C) \sum_{j = 1}^\infty \frac{1 - |a_n|^2}{|\xi - a_j|}.\] Since we assume $(a_j)$ is Frostman, at the singularities of $C$ we get a good bound on this sum.

Theorem \ref{movingzeros} should be compared with that of Matheson and Ross \cite{MR} who showed that every Frostman shift of a Frostman Blaschke product is Frostman; that is, if we start with a Frostman Blaschke product $B$ and we consider $\varphi_a \circ B$ where $\varphi_a(z) = (a-z)/(1 - \overline{a}z)$, then $\varphi_a \circ B$ is still a Frostman Blaschke product. We may think of this as saying that if we move the zeros of a Frostman Blaschke product in a systematic way (namely, to the places at which the Blaschke product assumes the value $a$), the resulting product is still Frostman. Their proof is based on a result of Tolokonnikov \cite{T} (that is itself based on a result of Pekarski \cite{P}) and a theorem of Hru\u s\u c\" ev and Vinogradov, \cite{HV}.

\begin{corollary} Let $(a_n)$ and $(z_n)$ be $\rho$-separated sequences with $\sup_n \rho(a_n, z_n) \le \lambda < 1$. Let $B$ and $C$ be the corresponding Blaschke products. Then $(a_n)$ is interpolating for $K_B^\infty$ if and only if $(z_n)$ is interpolating for $K_C^\infty$. \end{corollary}

\begin{proof} Suppose first that $(a_n)$ is interpolating for $K_B^\infty$. Since $(a_n)$ is then interpolating for $H^\infty$ and $(z_n)$ is $\rho$-separated with $\sup_n \rho(a_n, z_n) \le \lambda < 1$, it follows from Proposition~\ref{prop:near} that $(z_n)$ is interpolating for $H^\infty$. Similarly, the same is true if we interchange the roles of $z_n$ and $a_n$. The result now follows  from Hru\u s\u c\" ev and Vingogradov's work. (See also \cite[(1.12)]{D}.) 
\end{proof}

\textbf{Acknowledgments.}  Since August 2018, Pamela Gorkin has been serving as a Program Director in the Division of Mathematical Sciences
at the National Science Foundation (NSF), USA, and as a component of this position, she
received support from NSF for research, which included work on this paper. 

Brett D. Wick's research supported in part by NSF grants DMS-1800057 and DMS-1560955, as well as ARC DP190100970.

Any opinions, findings, and conclusions or recommendations expressed in this material are those of the
authors and do not necessarily reflect the views of the National Science Foundation.


\begin{thebibliography}{12}
\bibitem{AG} Akeroyd, John R.; Gorkin, Pamela Constructing Frostman-Blaschke products and applications to operators on weighted Bergman spaces. J. Operator Theory 74 (2015), no. 1, 149--175.

\bibitem{CMR} Cima, Joseph A.; Matheson, Alec L.; Ross, William T. The Cauchy transform. Mathematical Surveys and Monographs, 125. American Mathematical Society, Providence, RI, 2006.

\bibitem{C} Cohn, William, Radial limits and star invariant subspaces of bounded mean oscillation, Amer. J. Math 108, (1986), 719--749.

\bibitem{C1} Cohn, William, A maximum principle for star invariant subspaces, Houston J. Math 14 (1988), 23 - 37.

\bibitem{D} Dyakonov, Konstantin M., A free interpolation problem for a subspace of $H^\infty$, Bull. Lond. Math. Soc. 50 (2018), no. 3, 477--486.

\bibitem{G} Garnett, John B., Bounded Analytic Functions, Revised first edition. Graduate Texts in Mathematics, 236. Springer, New York, 2007.

\bibitem{H} Hoffman, Kenneth, Bounded analytic functions and Gleason parts. Ann. of Math. (2) 86 (1967), 74--111.

\bibitem{HV} Hru\u s\u c\" ev, S. V.; Vinogradov, S. A., Inner functions and multipliers of Cauchy type integrals. Ark. Mat. 19 (1981), no. 1, 23--42. 

\bibitem{MR} Matheson, Alec L.; Ross, William T., An observation about Frostman shifts. Comput. Methods Funct. Theory 7 (2007), no. 1, 111--126. 

\bibitem{M} Mortini Raymond;  Rupp, Rudolf,  A  Space Odyssey:
Extension Problems and Stable Ranks,
Accompanied by introductory chapters on point-set topology and Banach algebras, in preparation.

\bibitem{P} Pekarski\u{i}, A. A., Estimates of the derivative of a Cauchy-type integral with meromorphic density and their applications. (Russian) Mat. Zametki 31 (1982), no. 3, 389--402, 474.  

\bibitem{T}
Tolokonnikov, V. A., Blaschke products with the Carleson-Newman condition, and ideals of the algebra $H^\infty$. (Russian) Zap. Nauchn. Sem. Leningrad. Otdel. Mat. Inst. Steklov. (LOMI) 149 (1986), Issled. Line\u{i}n. Teor. Funktsi\u{i}. XV, 93--102, 188; translation in J. Soviet Math. 42 (1988), no. 2, 1603--1610.

\bibitem{T1} Tolokonnikov, V. A., Carleson's Blaschke products and Douglas algebras. (Russian) Algebra i Analiz 3 (1991), no. 4, 186--197; translation in St. Petersburg Math. J. 3 (1992), no. 4, 881--892. 

\bibitem{V}
Vinogradov S. A., Some remarks on free interpolation by bounded and slowly growing analytic functions, Zap. Nauchn. Sm. Leningrad. Otdel. Mat. Inst. Steklov. (LOMI) {\bf 126} (1983), 35 -- 46.

\bibitem{Zhu}
Zhu, Kehe Operator theory in function spaces. Second edition. Mathematical Surveys and Monographs, 138. American Mathematical Society, Providence, RI, 2007.

\end{thebibliography}
\end{document}